\documentclass[reqno]{amsart}
\usepackage{graphicx}
\usepackage{trajan}
\usepackage{multirow}
\usepackage{latexsym,array,delarray,amsthm,amssymb,epsfig}
\usepackage{amsmath}
\usepackage{yhmath}
\usepackage{mathdots}
\usepackage{MnSymbol}
\usepackage{tikz}
\usepackage{adjustbox}
\usepackage{rotating}
\usepackage{longtable}
\usepackage{enumitem}
\usepackage{caption}
\usepackage{wasysym}
\usetikzlibrary{shapes.misc,shadows}
\theoremstyle{plain}
\newtheorem{thm}{Theorem}[section]

\newtheorem*{thm*}{Theorem}
\newtheorem*{lemma*}{Lemma}
\newtheorem*{prop*}{Proposition}
\newtheorem*{cor*}{Corollary}
\newtheorem*{conj*}{Conjecture}

\theoremstyle{definition}

\theoremstyle{remark}

\newcommand{\cc}{\mathbb{C}}

\newcommand{\ca}{\mathcal{A}}

\newcommand{\rf}{\mathfrak{R}}

\begin{document}
\date{}

\title{On Classification of Four Dimensional Nilpotent Leibniz Algebras}
\author{Ismail Demir, Kailash C. Misra and Ernie Stitzinger}
\address{Department of Mathematics, North Carolina State University, Raleigh, NC 27695-8205}
\email{idemir@ncsu.edu, misra@ncsu.edu, stitz@ncsu.edu}
\thanks{KCM is partially supported by Simons Foundation grant \#  307555}

\begin{abstract}
Leibniz algebras are certain generalization of Lie algebras. In this paper we give the classification of four dimensional non-Lie nilpotent Leibniz algebras. We use the canonical forms for the congruence classes of matrices of bilinear forms and some other techniques to obtain our result.

\end{abstract}

\maketitle
\bigskip
\section{Introduction}

Leibniz algebras introduced by Loday \cite{lodayfr} are natural generalization of Lie algebras. Earlier, such algebraic structures had been considered by A. Bloh who called them $D$-algebras \cite{bloh}. In this paper we assume the field to be $\mathbb{C}$, the field of complex numbers. A left (resp. right) Leibniz algebra $A$ is a $\mathbb{C}$-vector space equipped with a bilinear product $[ \, , \,] : A \times A \longrightarrow A$ such that the left (resp. right) multiplication operator is a derivation. A left Leibniz algebra is not necessarily a right Leibniz algebra. In this paper, we consider (left) Leibniz algebras following Barnes \cite{barnes1}. Of course, a Lie algebra is a Leibniz algebra, but not conversely.  For a Leibniz algebra $A$, we define the ideals $A^1 = A $ and $A^i = [A, A^{i-1}]$ for $i \in \mathbb{Z}_{\geq 2}$. The Leibniz algebra $A$ is said be  abelian if $A^2 = 0$. It is nilpotent of class $c$ if $A^{c+1}=0$ but $A^c \neq 0$ for some positive integer $c$. An important abelian ideal of $A$ is $Leib(A)={\rm span}\{[a, a] \mid a\in A\}$. $A$ is a Lie algebra if and only if $Leib(A)= \{0\}$. The left center of $A$ is denoted by $Z^l(A)=\{x\in A\mid [x, a]=0 \  \rm{for \ all} \ a\in A\}$ and the right center of $A$ is denoted by $Z^r(A)=\{x\in A\mid [a, x]=0 \  \rm{for \ all} \ a\in A\}$. The center of $A$ is $Z(A)=Z^l(A)\cap Z^r(A)$. If $A$ is a nilpotent Leibniz algebra of class $c$ then $A^c\subseteq Z(A)$. $A$ is nilpotent and $\dim(Leib(A))=1$ implies $Leib(A)\subseteq Z(A)$. We say that a Leibniz algebra $A$ is split if it can be written as a direct sum of two nontrivial ideals. Otherwise $A$ is said to be non-split.  If $A$ is a non-split Leibniz algebra then $Z(A)\subseteq A^2$.

The classification of nilpotent Lie algebras is a very difficult problem. Due to lack of antisymmetry in Leibniz algebras the classification of non-Lie nilpotent Leibniz algebras is even harder. So far the complete classification of non-Lie nilpotent Leibniz algebras over $\cc$ of dimension less than or equal to four is known (see \cite{fourdimnil},  \cite{3dim}, \cite{cil}, \cite{our},  \cite{lodayfr}, \cite{3comp}). In this paper we revisit the classification of four dimensional non-Lie nilpotent Leibniz algebras using different techniques than that used in \cite{fourdimnil}. In particular, the authors use 
the classification of filiform Leibniz algebras of dimension four and the classification of unitary associative algebras of dimension five in \cite{fourdimnil}. We use congruence classes of bilinear forms as in (\cite{our}, \cite{our2}) for the classification of $4-$dimensional non-Lie nilpotent Leibniz algebras with $\dim(A^2)=2$ and $\dim(Leib(A))=1$. When $\dim(A^2)=1$ (resp. $\dim(A^2)=3$) result follows from  \cite{our2} (resp. \cite{exp}). For the remaining cases we use direct method. Our approach of using congruence classes of bilinear forms can be used to classify Leibniz algebras in higher dimensions which we plan to pursue in future.

\maketitle
\bigskip
\section{Classification of 4-Dimensional Nilpotent Leibniz Algebras}

Let $A$ be a $4-$dimensional non-Lie nilpotent Leibniz algebra. $A$ is non-Lie implies that $Leib(A)\neq 0$,  hence $A^2\neq0$. Then $\dim(A^2)=1, 2 \  \rm{or} \ 3$. It suffices to classify non-split non-Lie nilpotent Leibniz algebras.  The classification of $A$ when 
$\dim(A^2)=1, \ \rm{or} \ 3$ is known \cite{our2, exp}.

\begin{thm} (\cite{our2}  Theorem 2.2) Let $A$ be a $4-$dimensional non-split non-Lie nilpotent Leibniz algebra with $\dim(A^2)=1$. Then $A$ is isomorphic to a Leibniz algebra spanned by $\{x_1, x_2, x_3, x_4\}$ with the nonzero products given by one of  the following:
\begin{description}
\item[$\ca_1$] $[x_1, x_3]= x_4, [x_3, x_2]=x_4$.
\item[$\ca_2$] $[x_1, x_3]= x_4, [x_2, x_2]=x_4, [x_2, x_3]=x_4, [x_3, x_1]=x_4, [x_3, x_2]=-x_4$.
\item[$\ca_3$] $[x_1, x_2]=x_4, [x_2, x_1]=-x_4, [x_3, x_3]=x_4$.
\item[$\ca_4$] $[x_1, x_2]= x_4, [x_2, x_1]=-x_4, [x_2, x_2]=x_4, [x_3, x_3]=x_4$.
\item[$\ca_5$] $[x_1, x_2]= x_4, [x_2, x_1]=cx_4, [x_3, x_3]=x_4, \quad c\in \cc\backslash\{1, -1\}$.
\item[$\ca_6$] $[x_1, x_1]=x_4, [x_2, x_2]=x_4, [x_3, x_3]=x_4$.
\end{description}
\end{thm}

\begin{thm} (\cite{exp} Corollary 3) Let $A$ be a $4-$dimensional non-split non-Lie nilpotent Leibniz algebra with $\dim(A^2)=3$. Then $A$ is isomorphic to a Leibniz algebra spanned by $\{x_1, x_2, x_3, x_4\}$ with the nonzero products given by the following:
\\
$\ca_{7}: [x_1, x_1]=x_2, [x_1, x_2]=x_3, [x_1, x_3]=x_4.$
\end{thm}

It is left to consider the case $\dim(A^2)=2$. Since $A$ is nilpotent $\dim(A^3)=0$ or $\dim(A^3)=1$. Since $Leib(A)\subseteq A^2$ we have $\dim(Leib(A))=1$ or $\dim(Leib(A))=2$. Suppose $\dim(Leib(A))=1$.
Let $Leib(A)={\rm span}\{e_4\}$, and extend this to a basis $\{e_3, e_4\}$ for $A^2$. Let $V$ be a complementary subspace to $A^2$ in $A$ so that $A=A^2\oplus V$. Then for any $u, v\in V$, we have $[u, v]= \alpha e_3+ \alpha' e_4$ for some $\alpha , \alpha' \in \cc$. Define the bilinear form $f( \ , \ ): V\times V\rightarrow \cc$ by $f(u, v)= \alpha'$ for all $u, v\in V$. 
\par
Using ( \cite{congr}, Theorem 3.1), we choose a basis $\{e_1, e_2\}$ for $V$ so that the matrix $N$ of the bilinear form $f( \ , \ ): V\times V\rightarrow \cc$ is one the following:
\vskip 7pt
\noindent $(i) \left( \begin{array}{cc}
0&1 \\
-1&0
\end{array} \right), \quad
(ii) \left( \begin{array}{cc}
1&0 \\
0&0
\end{array} \right), \quad
(iii) \left( \begin{array}{cc}
1&0 \\
0&1
\end{array} \right), \quad
(iv) \left( \begin{array}{cc}
0&1 \\
-1&1
\end{array} \right), \quad
(v) \left( \begin{array}{cc}
0&1 \\
c&0
\end{array} \right)$\quad , 

\noindent 
\vskip 7pt ${\rm where} \ \ c \neq 1, -1$. However, since $e_4\in Leib(A)$ we observe that $N$ cannot be the matrix (i).

\begin{thm} Let $A$ be a $4-$dimensional non-split non-Lie nilpotent Leibniz algebra with $\dim(A^2)=2$, $\dim(A^3)=0$ and $\dim(Leib(A))=1$. Then $A$ is isomorphic to a Leibniz algebra spanned by $\{x_1, x_2, x_3, x_4\}$ with the nonzero products given by one of  the following:
\begin{description}
\item[$\ca_{8}$] $[x_1, x_1]=x_4, [x_1, x_2]=x_3=-[x_2, x_1]$.
\item[$\ca_{9}$] $[x_1, x_1]=x_4, [x_1, x_2]=x_3=-[x_2, x_1], [x_2, x_2]=x_4$.
\end{description}
\end{thm}

\begin{proof} Since $A$ is non-split nilpotent Leibniz algebra we have $A^2=Z(A)$. As $Leib(A)={\rm span}\{e_4\}$, and $A^2=Z(A)={\rm span}\{e_3, e_4\}$ we have $A=\rm span\{e_1, e_2, e_3, e_4\}$. 

\par If $N$ is the matrix $(ii)$ we have the nontrivial products in $A$ given by $[e_1, e_1]=e_4, [e_1, e_2]=\alpha e_3=-[e_2, e_1]$. Note that  $\alpha\in\cc\backslash\{0\}$ because $\dim(A^2)=1$ otherwise. Hence the base change $x_1=e_1, x_2=e_2, x_3=\alpha e_3, x_4=e_4$ shows that $A$ is isomorphic to the algebra $\ca_{8}$. If $N$ is the matrix $(iv)$ we have the nontrivial products in $A$ given by $[e_1, e_2]=\alpha e_3+e_4=-[e_2, e_1], [e_2, e_2]=e_4, \ \alpha\in\cc\backslash\{0\}.$ The change of basis $x_1=e_2, x_2=e_1, x_3=-\alpha e_3-e_4, x_4=e_4$ shows that $A$ is isomorphic to the algebra $\ca_{8}$ again. 

\par If $N$ is the matrix $(iii)$ we have the nontrivial products in $A$ given by $[e_1, e_1]=e_4, [e_1, e_2]=\alpha e_3=-[e_2, e_1], [e_2, e_2]=e_4, \ \alpha\in\cc\backslash\{0\}.$ Hence the base change $x_1=e_1, x_2=e_2, x_3=\alpha e_3, x_4=e_4$ shows that $A$ is isomorphic to the algebra $\ca_{9}$. If $N$ is the matrix $(v)$ then we have the nontrivial products in $A$ given by $[e_1, e_2]=\alpha e_3+e_4, [e_2, e_1]=-\alpha e_3+ce_4, \ \alpha\in\cc\backslash\{0\}.$ The base change $x_1=\frac{i}{2^{2/3}((1+c)^2\alpha^2)^{1/6}}e_1+ \frac{i(1+c)\alpha}{((1+c)^2\alpha^2)^{1/2}}e_2, x_2=-\frac{(1+c)\alpha}{2^{2/3}((1+c)^2\alpha^2)^{2/3}}e_1+e_2, x_3=\frac{i2^{1/3}\alpha}{((1+c)^2\alpha^2)^{1/6}}e_3-\frac{i(c-1)}{2^{2/3}((1+c)^2\alpha^2)^{1/6}}e_4, x_4=-\frac{((1+c)^2\alpha^2)^{1/3}}{2^{2/3}\alpha}e_4$ shows that $A$ is isomorphic to the algebra $\ca_{9}$ again.

\end{proof}

\begin{thm} Let $A$ be a $4-$dimensional non-split non-Lie nilpotent Leibniz algebra with $\dim(A^2)=2$ and $\dim(Leib(A))=1=\dim(A^3)$. Then $A$ is isomorphic to a Leibniz algebra spanned by $\{x_1, x_2, x_3, x_4\}$ with the nonzero products given by one of the following:
\begin{description}
\item[$\ca_{10}$] $[x_1, x_1]=x_4, [x_1, x_2]=x_3=-[x_2, x_1], [x_1, x_3]=x_4=-[x_3, x_1]$.
\item[$\ca_{11}$] $[x_1, x_2]=x_3=-[x_2, x_1], [x_2, x_2]=x_4, [x_1, x_3]=x_4=-[x_3, x_1]$.
\item[$\ca_{12}$] $[x_1, x_1]=x_4, [x_1, x_2]=x_3, [x_2, x_1]=-x_3+x_4, [x_1, x_3]=x_4=-[x_3, x_1]$.
\item[$\ca_{13}$] $[x_2, x_2]=x_4, [x_1, x_2]=x_3, [x_2, x_1]=-x_3+x_4, [x_1, x_3]=x_4=-[x_3, x_1]$.
\end{description}
\end{thm}

\begin{proof} Since $A$ is non-split nilpotent Leibniz algebra we have $A^3\subseteq Z(A)$ and $Z(A)\subseteq A^2$. Then $A^3=Z(A)$. Using $\dim(Leib(A))=1$ with the fact that $A$ is nilpotent we get $Leib(A)=Z(A)=A^3={\rm span}\{e_4\}$. We have $A^2={\rm span}\{e_3, e_4\}$ and $A={\rm span}\{e_1, e_2, e_3, e_4\}$.  Consider the following Leibniz identities:
\begin{equation} \label{leib1}
\begin{cases}
[e_1, [e_2, e_1]]=[[e_1, e_2], e_1]+ [e_2, [e_1, e_1]] \\
[e_2, [e_1, e_2]]=[[e_2, e_1], e_2]+ [e_1, [e_2, e_2]] \\
[e_3, [e_1, e_2]]=[[e_3, e_1], e_2]+ [e_1, [e_3, e_2]]
\end{cases}
\end{equation}

If $N$ is the matrix $(ii)$ then by the Leibniz identities (\ref{leib1}) the nontrivial products in $A$ are given by 
\begin{equation} 
[e_1, e_1]=e_4, [e_1, e_2]=\alpha e_3=-[e_2, e_1], [e_1, e_3]=\beta e_4=-[e_3, e_1], [e_2, e_3]=\gamma e_4=-[e_3, e_2].      \label{alg3}
\end{equation}
Here $\beta, \gamma \in\cc$ and $\alpha\in\cc\backslash\{0\}$ since $\dim(A^2)=2$. Suppose $\gamma=0$. Then $\beta\neq0$ since $A^3\neq0$. The base change $x_1=e_1, x_2=\frac{1}{\alpha\beta}e_2, x_3=\frac{1}{\beta}e_3, x_4=e_4$ yields that $A$ is isomorphic to the algebra $\ca_{10}$. Now suppose $\gamma\neq0$. Then the base change $x_1=e_2, x_2=-\alpha\gamma e_1+\alpha\beta e_2, x_3=\gamma\alpha^2e_3, x_4=(\alpha\gamma)^2e_4$ shows that $A$ is isomorphic to the algebra $\ca_{11}$. If $N$ is the matrix $(iv)$ then by the Leibniz identities (\ref{leib1}) the nontrivial products in $A$ are given by $[e_1, e_2]=\alpha e_3+ e_4=-[e_2, e_1], [e_2, e_2]=e_4, [e_1, e_3]=\beta e_4=-[e_3, e_1], [e_2, e_3]=\gamma e_4=-[e_3, e_2], \ \beta, \gamma \in\cc, \alpha\in\cc\backslash\{0\}$. Then the following change of basis $x_1=e_2, x_2=e_1, x_3=-\alpha e_3-e_4, x_4=e_4$ gives that $A$ is isomorphic to the following algebra $[x_1, x_1]=x_4, [x_1, x_2]=x_3=-[x_2, x_1], [x_1, x_3]=-\alpha\gamma x_4=-[x_3, x_1], [x_2, x_3]=-\alpha\beta x_4=-[x_3, x_2].$ This algebra is isomorphic to an algebra with the nontrivial products given by (\ref{alg3}). Therefore, $A$ is isomorphic to (\ref{alg3}), and hence, as above  isomorphic to $\ca_{10}$ or $\ca_{11}$.

\par If $N$ is the matrix $(iii)$ then by the Leibniz identities (\ref{leib1}) the nontrivial products in $A$ are given by
\begin{equation}
[e_1, e_1]=e_4=[e_2, e_2], [e_1, e_2]=\alpha e_3=-[e_2, e_1], [e_1, e_3]=\beta e_4=-[e_3, e_1], [e_2, e_3]=\gamma e_4=-[e_3, e_2]. \label{alg4}
\end{equation}
Here $\beta, \gamma \in\cc$ and $\alpha\in\cc\backslash\{0\}$ since $\dim(A^2)=2$. Let $\beta^2+\gamma^2=0$. Note that $A^3\neq0$ implies that $\beta, \gamma\neq0$. The change of basis $x_1=-\frac{2\gamma}{\beta^2\alpha}e_1, x_2=-\frac{\gamma}{\beta^2\alpha}e_1+ \frac{1}{\beta\alpha}e_2, x_3=-\frac{2\gamma}{\beta^3\alpha}e_3+\frac{2\gamma^2}{\beta^4\alpha^2}e_4, x_4=\frac{4\gamma^2}{\beta^4\alpha^2}e_4$ shows that $A$ is isomorphic to the algebra $\ca_{12}$. Now let $\beta^2+\gamma^2\neq0$. The base change $x_1=\frac{2\gamma-2i\beta}{\alpha(\beta^2+\gamma^2)}e_1-\frac{2\beta+2i\gamma}{\alpha(\beta^2+\gamma^2)}e_2, x_2=\frac{4\gamma}{\alpha(\beta^2+\gamma^2)}e_1-\frac{4\beta}{\alpha(\beta^2+\gamma^2)}e_2-\frac{4i}{\alpha(\beta^2+\gamma^2)}e_3, x_3=\frac{8i}{\alpha(\beta^2+\gamma^2)}e_3, x_4=\frac{16}{\alpha^2(\beta^2+\gamma^2)}e_4$ yields that $A$ is isomorphic to the algebra $\ca_{13}$. If $N$ is the matrix $(v)$ then by the Leibniz identities (\ref{leib1}) we have the nontrivial products in $A$ given by $[e_1, e_2]=\alpha e_3+ e_4, [e_2, e_1]=-\alpha e_3+ ce_4, [e_1, e_3]=\beta e_4=-[e_3, e_1], [e_2, e_3]=\gamma e_4=-[e_3, e_2],  \ \beta, \gamma \in\cc, \alpha\in\cc\backslash\{0\}$. The following change of basis $x_1=-ie_1+ie_2, x_2=e_1+e_2, x_3=-2ie_3+\frac{i(c-1)}{\alpha}e_4, x_4=(c+1)e_4$  shows that $A$ is isomorphic to an algebra with the notrivial products given by $[x_1, x_1]=x_4=[x_2, x_2], [x_1, x_2]=\alpha x_3=-[x_2, x_1], [x_1, x_3]=\frac{-2\beta+2\gamma}{c+1} x_4=-[x_3, x_1], [x_2, x_3]=\frac{-2i\beta-2i\gamma}{c+1} x_4=-[x_3, x_2].$ This algebra is isomorphic to an algebra with nontrivial products given in  (\ref{alg4}). Therefore, as above $A$ is isomorphic to $\ca_{12}$ or $\ca_{13}$.
\end{proof}

The remaining cases can be completed via direct method. 

\begin{thm} Let $A$ be a $4-$dimensional non-split non-Lie nilpotent Leibniz algebra with  $\dim(A^2)=2=\dim(Leib(A))$ and $\dim(A^3)=0$. Then $A$ is isomorphic to a Leibniz algebra spanned by $\{x_1, x_2, x_3, x_4\}$ with the nonzero products given by the following:
\begin{description}
\item[$\ca_{14}$] $[x_1, x_1]=x_3, [x_1, x_2]=x_4.$
\item[$\ca_{15}$] $[x_1, x_1]=x_3, [x_2, x_1]=x_4.$
\item[$\ca_{16}$] $[x_1, x_2]=x_4, [x_2, x_1]=x_3, [x_2, x_2]=-x_3.$
\item[$\ca_{17}$] $[x_1, x_1]=x_3, [x_1, x_2]=x_4, [x_2, x_1]=\alpha x_4, \quad \alpha\in \cc\backslash\{-1, 0\}.$
\item[$\ca_{18}$] $[x_1, x_1]=x_3, [x_2, x_1]=x_4, [x_1, x_2]=\alpha x_3, [x_2, x_2]=-x_4, \quad \alpha\in \cc\backslash\{-1\}.$
\item[$\ca_{19}$] $[x_1, x_1]=x_3, [x_1, x_2]=x_3, [x_2, x_1]=x_3+ x_4, [x_2, x_2]=x_4.$
\end{description}
\end{thm}

\begin{proof} Since $A$ is non-split nilpotent Leibniz algebra we have $A^2=Z(A)$. Choose $Leib(A)=A^2=Z(A)={\rm span}\{e_3, e_4\}$. Let $H$ be a maximal abelian subalgebra of $A$ containing $Leib(A)$. Then $\rm dim(H)=2 \ \rm or \ 3$. Notice that if $\dim(Z^l(A))=3=\dim(Z^r(A))$ then $\dim(A^2)=1$ which is a contradiction. Therefore, we can split the problem into three cases:

\par Let $\dim(Z^l(A))=3$ and $\dim(Z^r(A))=2$. We have $Leib(A)=A^2=Z(A)=Z^r(A)={\rm span}\{e_3, e_4\}$. Extend this to a basis $\{e_2, e_3, e_4\}$ for $Z^l(A)$. Choose $e_1\in A$ such that $[e_1, e_1]=e_3$. Then the nontrivial products in $A={\rm span}\{e_1, e_2, e_3, e_4\}$ are given by $[e_1, e_1]=e_3, [e_1, e_2]=\alpha_1e_3+ \alpha_2e_4$. Here $\alpha_2\neq0 \  \rm since \ \dim(A^2)=2.$ The change of basis $x_1=e_1, x_2=e_2, x_3=e_3, x_4=\alpha_1e_3+ \alpha_2e_4$ gives that $A$ is isomorphic to the algebra $\ca_{14}$.

\par Let $\dim(Z^l(A))=2$ and $\dim(Z^r(A))=3$. We have $Leib(A)=A^2=Z(A)=Z^l(A)={\rm span}\{e_3, e_4\}$. Extend this to a basis $\{e_2, e_3, e_4\}$ for $Z^r(A)$. Choose $e_1\in A$ such that $[e_1, e_1]=e_3$. Then the nontrivial products in $A={\rm span}\{e_1, e_2, e_3, e_4\}$ are given by $[e_1, e_1]=e_3, [e_2, e_1]=\alpha_1e_3+ \alpha_2e_4 \quad  (\alpha_2\neq0 \  \rm since \ \dim(A^2)=2)$. From the base change  $x_1=e_1, x_2=e_2, x_3=e_3, x_4=\alpha_1e_3+ \alpha_2e_4$ we see that $A$ is isomorphic to the algebra $\ca_{15}$.

\par Let $\dim(Z^l(A))=2=\dim(Z^r(A))$ and $\dim(H)=3$. We have $Leib(A)=A^2=Z(A)=Z^l(A)=Z^r(A)={\rm span}\{e_3, e_4\}$. Extend this to a basis $\{e_2, e_3, e_4\}$ for $H$. Choose $e_1\in A$ such that $[e_1, e_1]=e_3$. Then the nontrivial products in $A={\rm span}\{e_1, e_2, e_3, e_4\}$ are given by $[e_1, e_1]=e_3, [e_1, e_2]=\alpha_1e_3+\alpha_2e_4, [e_2, e_1]=\beta_1e_3+\beta_2e_4.$
Note that $\alpha_2+\beta_2\neq0$ since $\dim(Leib(A))=2$. Let $\alpha_1\beta_2\neq\alpha_2\beta_1$. Then the base change $x_1=e_2, x_2=\frac{\alpha_2\beta_1-\alpha_1\beta_2}{\alpha_2+\beta_2}e_1-\frac{\alpha_2}{\alpha_2+\beta_2}e_2, x_3=\frac{\alpha_2\beta_1-\alpha_1\beta_2}{\alpha_2+\beta_2}\alpha_1e_3+\frac{\alpha_2\beta_1-\alpha_1\beta_2}{\alpha_2+\beta_2}\alpha_2e_4, x_4=\frac{\alpha_2\beta_1-\alpha_1\beta_2}{\alpha_2+\beta_2}\beta_1e_3+\frac{\alpha_2\beta_1-\alpha_1\beta_2}{\alpha_2+\beta_2}\beta_2e_4$ gives that $A$ is isomorphic to the algebra $\ca_{16}$.  Let $\alpha_1\beta_2=\alpha_2\beta_1$. Note that if $\alpha_2=0$ then $\alpha_1\beta_2=0$. Hence $\beta_2\neq0$(since $\alpha_2+\beta_2\neq0$) implies that $\alpha_1=0$. But then $\dim(Z^r(A))=3$ which is a contradiction. Similarly if $\beta_2=0$ then $\dim(Z^l(A))=3$ which is a contradiction. Hence $\alpha_2, \beta_2\neq0$. Then the change of basis $x_1=e_1, x_2=e_2,  x_3=e_3,  x_4=\alpha_1e_3+\alpha_2e_4$ shows that $A$ is isomorphic to the algebra $\ca_{17}$.

\par Let $\dim(Z^l(A))=2=\dim(Z^r(A))=\dim(H)$. Then we have $H=Leib(A)=A^2=Z(A)=Z^l(A)=Z^r(A)={\rm span}\{e_3, e_4\}$. Choose $e_1, e_2\in A$ such that $[e_1, e_1]=e_3, [e_2, e_2]=e_4$. Then the products in $A={\rm span}\{e_1, e_2, e_3, e_4\}$ are given by $[e_1, e_1]=e_3, [e_1, e_2]=\alpha_1e_3+\alpha_2e_4, [e_2, e_1]=\beta_1e_3+\beta_2e_4, [e_2, e_2]=e_4.$ 
\\
Suppose  $\alpha_2=0$. Then the nonzero products in $A$ are given by:
\begin{equation} \label{alg1}
[e_1, e_1]=e_3, [e_1, e_2]=\alpha_1e_3, [e_2, e_1]=\beta_1e_3+\beta_2e_4, [e_2, e_2]=e_4. 
\end{equation}
If $\alpha_1=0$ then we have the nonzero products in $A$ become:
\begin{equation} \label{alg2}
[e_1, e_1]=e_3, [e_2, e_1]=\beta_1e_3+\beta_2e_4, [e_2, e_2]=e_4. 
\end{equation}
If $\beta_1\neq0$ then $\beta_1\beta_2\neq1$ since otherwise $\dim(H)=3$.  In this case the change of basis $x_1=-\beta_1e_1+e_2,  x_2=\beta_1e_1,  x_3=(1-\beta_1\beta_2)e_4,  x_4=-\beta^2_1e_3$
shows that $A$ is isomorphic to the algebra $\ca_{18}$ where $\alpha=\frac{\beta_1\beta_2}{1-\beta_1\beta_2}\in \cc\backslash\{-1\}$. If $\beta_1=0$ then $\beta_2\neq0$ since $A$ is non-split. In such case the base change $x_1=\frac{1}{\beta_2}e_1,  x_2=-e_2,  x_3=\frac{1}{\beta^2_2}e_3,  x_4=-e_4$ shows that $A$ is isomorphic to the algebra $\ca_{18}$ where $\alpha=0$. Now suppose $\alpha_1\neq0$. If $\alpha_1\beta_2=1$, then $\beta_1\neq0$, because $\beta_1=0$ implies $\dim(Z^r(A))=3$. The change of basis $x_1=\alpha^{1/3}_1\beta^{1/6}_1e_1, x_2=\frac{\alpha^{5/6}_1+\alpha^{1/3}_1\beta^{1/2}_1}{\beta^{1/3}_1}e_1-\frac{1}{\alpha^{1/6}_1\beta^{1/3}_1}e_2, x_3=\alpha^{2/3}_1\beta^{1/3}_1e_3, x_4=\frac{\alpha^{7/6}_1-\alpha^{1/6}_1\beta_1}{\beta^{1/6}_1}e_3-\frac{1}{\alpha^{5/6}_1\beta^{1/6}_1}e_4$ shows that $A$ is isomorphic to the algebra $\ca_{19}$. If $\alpha_1\beta_2\neq1$, then from the base change $x_1=\alpha_1e_1, x_2=\alpha_1e_1-e_2, x_3= \alpha^2_1e_3, x_4= -\alpha_1\beta_1e_3+(1-\alpha_1\beta_2)e_4$ we see that $A$ is isomorphic to an algebra with the nonzero products given by (\ref{alg2}). Hence, as above $A$ is isomorphic to 
$\ca_{18}$.
\\ 
Now suppose $\alpha_2\neq0$. If $(\alpha_1+\beta_1)(\alpha_2+\beta_2)=1$ then $\dim(H)=3$. Hence $(\alpha_1+\beta_1)(\alpha_2+\beta_2)\neq1$. Suppose $\beta_2=0$. If $\beta_1=0$ then with the change of basis 
$x_1=e_2, x_2=e_1, x_3=e_4, x_4=e_3$ we see that $A$ is isomorphic to an algebra with nonzero products given by (\ref{alg2}), hence as before $A$ is isomorphic to 
$\ca_{18}$. So assume $\beta_1\neq0$. If $\alpha_2\beta_1=1$ then $\alpha_1\alpha_2\neq0$, so $\alpha_1\neq0$. Then the base change $x_1=\frac{1}{\alpha^{1/3}_1\alpha^{5/6}_2}e_1-\frac{\alpha^{1/6}_2}{\alpha^{1/3}_1}e_2, x_2=\frac{1+(\alpha_1\alpha_2)^{1/2}}{\alpha^{1/3}_1\alpha^{5/6}_2}e_1-\frac{\alpha^{1/6}_2}{\alpha^{1/3}_1}e_2, x_3=-\frac{\alpha^{1/3}_1}{\alpha^{2/3}_2}e_3, x_4=\frac{1-\alpha_1\alpha_2}{\alpha^{1/6}_1\alpha^{7/6}_2}e_3-\frac{\alpha^{5/6}_2}{\alpha^{1/6}_1}e_4$ shows that $A$ is isomorphic to the algebra $\ca_{19}$. If $\alpha_2\beta_1\neq1$, then the change of basis $x_1=\beta_1e_1-e_2,  x_2=e_1, x_3=-\alpha_1\beta_1e_3+(1-\beta_1\alpha_2)e_4,  x_4=e_3$ shows that $A$ is isomorphic to an algebra with the nonzero products given by (\ref{alg2}). So as before $A$ is isomorphic to  $\ca_{18}$. Now suppose $\beta_2\neq0$. If $\beta_1=0$ the base change $x_1=e_2, x_2=e_1, x_3=e_4, x_4=e_3$ yields that $A$ is isomorphic to an algebra with the nonzero products given by (\ref{alg1}), hence as above $A$ is isomorphic to $\ca_{18}$ or $\ca_{19}$.
If $\beta_1\neq0$ then the change of basis $x_1=-\frac{1+s\beta_2}{\alpha_2}e_1+e_2, x_2=se_1+e_2, x_3=[(\frac{1+s\beta_2}{\alpha_2})^2-\frac{(\alpha_1+\beta_1)(1+s\beta_2)}{\alpha_2}]e_3+[1-\frac{(\alpha_2+\beta_2)(1+s\beta_2)}{\alpha_2}]e_4, x_4=s(s+\alpha_1+\beta_1)e_3+ (1+s(\alpha_2+\beta_2))e_4$ where $s=\frac{-1-\alpha_1\beta_2+\alpha_2\beta_1+\sqrt{(1+\alpha_1\beta_2-\alpha_2\beta_1)^2-4\alpha_1\beta_2}}{2\beta_2}$ shows that $A$ is isomorphic to an algebra with nonzero products given by (\ref{alg2}), hence $A$ is isomorphic to $\ca_{18}$ as above.

\end{proof}

\begin{thm} Let $A$ be a $4-$dimensional non-split non-Lie nilpotent Leibniz algebra with $\dim(A^2)=2=\dim(Leib(A))$ and $dim(A^3)=1$. Then $A$ is isomorphic to a Leibniz algebra spanned by $\{x_1, x_2, x_3, x_4\}$ with the nonzero products given by the following:
\begin{description}
\item[$\ca_{20}$] $[x_1, x_2]=x_3, [x_1, x_3]=x_4$.
\item[$\ca_{21}$] $[x_1, x_2]=x_3, [x_2, x_2]=x_4, [x_1, x_3]=x_4$.
\item[$\ca_{22}$] $[x_1, x_2]=x_3, [x_2, x_1]=x_4, [x_1, x_3]=x_4$.
\item[$\ca_{23}$] $[x_1, x_2]=x_3, [x_2, x_1]=x_4, [x_2, x_2]=x_4, [x_1, x_3]=x_4$.
\item[$\ca_{24}$] $[x_1, x_1]=x_3, [x_2, x_1]=x_4, [x_1, x_3]=x_4$.
\item[$\ca_{25}$] $[x_1, x_1]=x_3, [x_2, x_2]=x_4, [x_1, x_3]=x_4$.
\end{description}
\end{thm}

\begin{proof} Since $A$ is non-split nilpotent Leibniz algebra we have $A^3\subseteq Z(A)$ and $Z(A)\subseteq A^2$. Then $A^3=Z(A)$. Choose $A^3=Z(A)={\rm span}\{e_4\}$. Extend this to a basis $\{e_3, e_4\}$ of $A^2=Leib(A)$. Then the products in $A={\rm span}\{e_1, e_2,  e_3, e_4\}$ are given by $[e_1, e_1]=\alpha_1e_3+\alpha_2e_4, [e_1, e_2]=\alpha_3e_3+\alpha_4e_4, [e_2, e_1]=\beta_1e_3+\beta_2e_4, [e_2, e_2]=\beta_3e_3+\beta_4e_4, [e_1, e_3]=\gamma_1e_4, [e_2, e_3]=\gamma_2e_4.$
From the Leibniz identities
$$
[e_1, [e_2, e_1]]=[[e_1, e_2], e_1]+ [e_2, [e_1, e_1]] \ \  \rm{and} \ \ [e_2, [e_1, e_2]]=[[e_2, e_1], e_2]+ [e_1, [e_2, e_2]] 
$$
we obtain the following equations:
\begin{equation} \label{leib2}
\begin{cases}
\gamma_1\beta_3=\alpha_1\gamma_2 \\
\gamma_2\alpha_3=\beta_1\gamma_1
\end{cases}
\end{equation}

\par Let $\gamma_2=0$. Then $\gamma_1\neq0$ since $A^3\neq0$. From (\ref{leib2}) we have $\beta_1=0=\beta_3$. Hence we have the following products in $A$:
\begin{equation} \label{alg5}
[e_1, e_1]=\alpha_1e_3+\alpha_2e_4, [e_1, e_2]=\alpha_3e_3+\alpha_4e_4, [e_2, e_1]=\beta_2e_4, [e_2, e_2]=\beta_4e_4, [e_1, e_3]=\gamma_1e_4.
\end{equation}
Let $\alpha_1=0$. Then $\alpha_3\neq0$ since $A^2\neq0$. Then we have the following products in $A$:
\begin{equation} \label{alg6}
[e_1, e_1]=\alpha_2e_4, [e_1, e_2]=\alpha_3e_3+\alpha_4e_4, [e_2, e_1]=\beta_2e_4, [e_2, e_2]=\beta_4e_4, [e_1, e_3]=\gamma_1e_4.
\end{equation}
Suppose $\alpha_2=0$. Then the products in $A$ are the following:
\begin{equation} \label{alg7}
[e_1, e_2]=\alpha_3e_3+\alpha_4e_4, [e_2, e_1]=\beta_2e_4, [e_2, e_2]=\beta_4e_4, [e_1, e_3]=\gamma_1e_4.
\end{equation}
If $\beta_2=0=\beta_4$ the base change $x_1=e_1, x_2=e_2, x_3=\alpha_3e_3+\alpha_4e_4, x_4=\alpha_3\gamma_1e_4$ shows that $A$ is isomorphic to the algebra $\ca_{20}$. If $\beta_2=0$ and $\beta_4\neq0$ the base change $x_1=e_1, x_2=\frac{\alpha_3\gamma_1}{\beta_4}e_2, x_3=\frac{\alpha_3\gamma_1}{\beta_4}(\alpha_3e_3+\alpha_4e_4), x_4=\frac{(\alpha_3\gamma_1)^2}{\beta_4}e_4$ yields that $A$ is isomorphic to the algebra $\ca_{21}$. If $\beta_2\neq0$ and $\beta_4=0$ the base change $x_1=\frac{\beta_2}{\alpha_3\gamma_1}e_1, x_2=e_2, x_3=\frac{\beta_2}{\alpha_3\gamma_1}(\alpha_3e_3+\alpha_4e_4), x_4=\frac{\beta^2_2}{\alpha_3\gamma_1}e_4$ yields that $A$ is isomorphic to the algebra $\ca_{22}$. Finally, if $\beta_2\neq0$ and $\beta_4\neq0$ then the base change $x_1=\frac{\beta_2}{\alpha_3\gamma_1}e_1, x_2=\frac{\beta^2_2}{\alpha_3\gamma_1\beta_4}e_2, x_3=\frac{\beta^3_2}{\alpha^2_3\gamma^2_1\beta_4}(\alpha_3e_3+\alpha_4e_4), x_4=\frac{\beta^4_2}{\alpha^2_3\gamma^2_1\beta_4}$ shows that $A$ is isomorphic to the algebra $\ca_{23}$.
\\
Now suppose $\alpha_2\neq0$. Then the base change $x_1=\gamma_1e_1-\alpha_2e_3, x_2=e_2, x_3=e_3, x_4=e_4$ shows that $A$ is isomorphic to an algebra with nonzero products given by (\ref{alg7}), hence as above $A$ is isomorphic to $\ca_{20}$, $\ca_{21}$, $\ca_{22}$ or $\ca_{23}$.
\\
Now let $\alpha_1\neq0$. Suppose $\alpha_3=0$. If $\alpha_4=0$ then the nonzero products in $A$ are given by:
\begin{equation} \label{alg8}
[e_1, e_1]=\alpha_1e_3+\alpha_2e_4, [e_2, e_1]=\beta_2e_4, [e_2, e_2]=\beta_4e_4, [e_1, e_3]=\gamma_1e_4.
\end{equation} 
If $\beta_4=0$ then $\beta_2\neq0$ since $A$ is non-split. The base change $x_1=e_1, x_2=\frac{\alpha_1\gamma_1}{\beta_2}e_2, x_3=\alpha_1e_3+\alpha_2e_4, x_4=\alpha_1\gamma_1e_4$ shows that $A$ is isomorphic to the algebra $\ca_{24}$. If $\beta_4\neq0$ the base change $x_1=e_1-\frac{\beta_2}{\beta_4}e_2, x_2=(\frac{\alpha_1\gamma_1}{\beta_4})^{1/2}e_2+\frac{\alpha^{1/2}_1\beta_2}{(\beta_4\gamma_1)^{1/2}}e_3, x_3=\alpha_1e_3+\alpha_2e_4, x_4=\alpha_1\gamma_1e_4$ shows that $A$ is isomorphic to the algebra $\ca_{25}$.
Now assume $\alpha_4\neq0$. Then using the base change $x_1=e_1, x_2=\gamma_1e_2-\alpha_4e_3, x_3=e_3, x_4=e_4$ we see that $A$ is isomorphic to an algebra with nonzero products given by (\ref{alg8}). Hence, as above $A$ is isomorphic to $\ca_{24}$ or $\ca_{25}$. Suppose $\alpha_3\neq0$. Then using the change of basis $x_1=\alpha_3e_1-\alpha_1e_2, x_2=e_2, x_3=e_3, x_4=e_4$ we see that $A$ is isomorphic to an algebra with nonzero products given by (\ref{alg6}). Therefore, as above $A$ is isomorphic to $\ca_{20}$, $\ca_{21}$, $\ca_{22}$, $\ca_{23}$, $\ca_{24}$ or $\ca_{25}$.

\par Let $\gamma_2\neq0$.  If $\gamma_1=0$ (resp. $\gamma_1 \neq 0$) then using the base change 
$x_1=e_2, x_2=e_1, x_3=e_3, x_4=e_4$ (resp. $x_1=e_1, x_2=\gamma_2e_1-\gamma_1e_2, x_3=e_3, x_4=e_4$) we see that $A$ is isomorphic to an algebra with nonzero products given by 
(\ref{alg5}). Hence,  as above $A$ is isomorphic to $\ca_{20}$, $\ca_{21}$, $\ca_{22}$ or $\ca_{23}$..

\end{proof}

In summary, if $A$ is non-split non-Lie nilpotent Leibniz algebra of dimension four, then it is isomorphic to one of the isomorphism classes: $\ca_1, \ca_2, \ca_3, \ldots, \ca_{24}, \ca_{25}$. Using  the algorithm in Mathematica  given in \cite{cil}, we verified that these isomorphism classes are pairwise nonisomorphic. In \cite{fourdimnil}, there are 21 isomorphism classes but it can be seen that $\rf_4(\alpha)$ is isomorphic to $\ca_{23}$ (resp. $\ca_{21}$) if $\alpha = 0$ (resp. $\alpha = 1)$; $\rf_{14}(\alpha)$ is isomorphic to $\ca_6$ (resp. $\ca_5$) if $\alpha = 0$ (resp. $\alpha \neq 0$); $\rf_{20}(\alpha)$ is isomorphic to $\ca_{14}$ (resp. $\ca_{17}$) if $\alpha = 0$ (resp. $\alpha \neq 0 , 1$). Furthermore, the isomorphism class $\ca_1$ is not contained in their list which seems to be an error.

\end{document}